\documentclass[11pt]{article}
\usepackage{latexsym,amssymb,amsmath,amsthm,enumerate,geometry,float,cite,tikz,pgfplots}
\usepackage{verbatim}
\newtheorem{theorem}{Theorem}
\newtheorem{proposition}[theorem]{Proposition}

\newtheorem{problem}[theorem]{Problem}

\newtheorem{claim}{Claim}
\usepackage{lineno}
\usepackage{setspace}
\allowdisplaybreaks

\begin{document}
\onehalfspace

\title{Sparse vertex cutsets and the maximum degree}
\author{St\'{e}phane Bessy$^1$ \and Johannes Rauch$^2$ \and Dieter Rautenbach$^2$ \and U\'{e}verton S. Souza$^3$}
\date{}

\maketitle
\vspace{-10mm}
\begin{center}
{\small 
$^1$ LIRMM, Univ Montpellier, CNRS, Montpellier, France\\
$^2$ Institute of Optimization and Operations Research, Ulm University, Ulm, Germany\\
$^3$ Instituto de Computa\c{c}\~{a}o, Universidade Federal Fluminense, Niter\'{o}i, Brazil\\
\texttt{stephane.bessy@lirmm.fr},
\texttt{$\{$johannes.rauch,dieter.rautenbach$\}$@uni-ulm.de},
\texttt{ueverton@ic.uff.br}
}
\end{center}

\begin{abstract}
We show that every graph $G$ 
of maximum degree $\Delta$ and sufficiently large order 
has a vertex cutset $S$ of order at most $\Delta$ 
that induces a subgraph $G[S]$ of maximum degree at most $\Delta-3$. 
For $\Delta\in \{ 4,5\}$, we refine this result by considering also 
the average degree of $G[S]$.
If $G$ has no $K_{r,r}$ subgraph, 
then we show the existence of a vertex cutset 
that induces a subgraph of maximum degree at most 
$\left(1-\frac{1}{{r\choose 2}}\right)\Delta+O(1)$.\\[3mm]
{\bf Keywords:} Fragile graph; stable cutset; matching cutset
\end{abstract}

\section{Introduction}

Answering a question posed by Caro, 
it was shown by Chen and Yu \cite{chyu} that every graph 
with $n$ vertices and at most $2n-4$ edges has
an independent (vertex) cutset.
Chen, Faudree, and Jacobson \cite{chfaja} showed that
the smallest such cutsets may be arbitrarily large as the number of edges approaches $2n$ but imposing a slightly stronger bound on the number of edges,
one can guarantee the existence of small independent cutsets.
Le and Pfender \cite{lepf} characterized the graphs with $n$ vertices 
and $2n-3$ edges that do not have an independent cutset.
The algorithmic problem of deciding the existence of independent cutsets
in a given graph was considered \cite{lera,lemomu}
with a particular focus on line graphs 
because independent cutsets in the line graph of some graph $G$
correspond to matching (edge) cutsets in $G$ \cite{bofapr,chhslelepe,ch,gosa}.
For subcubic graphs of order at least $8$,
the result of Chen and Yu implies the existence of an independent cutset,
while the complete graph $K_4$ 
and the triangular prism 
are cubic graphs of orders $4$ and $6$ with no such cutsets.
In other words, a sufficiently large graph of maximum degree at most $3$
has an independent cutset.
Motivated by this observation,
we consider the existence of sparse cutsets 
in sufficiently large graphs of bounded maximum degree.

We consider only finite, simple, and undirected graphs,
and use standard terminology.
Let $G$ be a graph and let $S$ be a set of vertices of $G$.
Let $G[S]$ denote the subgraph of $G$ induced by $S$
and let $G-S=G[V(G)\setminus S]$.
The set $S$ is a cutset of $G$ if $G-S$ is disconnected.
Let $\Delta_G(S)$ and $\bar{d}_G(S)$ denote 
the maximum degree and the average degree of $G[S]$, respectively.
If $S$ is a minimal cutset in a graph $G$ of maximum degree at most $\Delta$,
then every vertex in $S$ has at least one neighbor in every component of $G-S$,
which implies the trivial bound 
\begin{eqnarray}\label{e1}
\Delta_G(S)&\leq& \Delta-2.
\end{eqnarray}

We believe that this can be improved and pose the following.

\begin{problem}\label{problem1}
Are there two functions 
$f:\mathbb{N}\to \mathbb{N}$
and
$g:\mathbb{N}\to \mathbb{N}$
with $\lim\limits_{\Delta\to\infty}f(\Delta)=\infty$
such that 
every connected graph $G$ of order at least $g(\Delta)$ and 
maximum degree at most $\Delta$ 
has a cutset $S$ with $\Delta_G(S)\leq \Delta-f(\Delta)$?
\end{problem}
The following construction shows that $f(\Delta)=O\left(\sqrt{\Delta}\log(\Delta)\right)$:
Axenovich, Sereni, Snyder, and Weber \cite{axsesnwe}
showed the existence of some positive constant $c$ 
such that, for every positive integers $n'$ and $\Delta'$ with $n'\geq \Delta'$,
there is a bipartite graph $H(n',\Delta')$ of maximum degree at most $\Delta'$
whose partite sets $A$ and $B$ both have order $n'$ 
with the property that 
$$\min\Big\{ |I\cap A|,|I\cap B|\Big\}\leq c\frac{\log (\Delta')}{\Delta'}n'
\mbox{ for every independent set $I$ in $H(n',\Delta')$.}$$
Now, let $\Delta$ be an integer at least $9$.
Let $n'=\Delta+1-2\left\lceil\sqrt{\Delta}\right\rceil$
and $\Delta'=\left\lceil\sqrt{\Delta}\right\rceil$.
Let the graph $G$ arise from a path or cycle $G_0$ of order at least $3$ 
by replacing every vertex $u$ of $G_0$ 
with a clique $K_u$ of order $n'$ and
replacing every edge $uv$ of $G_0$ 
with a copy $H_{uv}$ of $H(n',\Delta')$.
By construction, the graph $G$ has maximum degree at most $\Delta$.
Furthermore, if $S$ is a cutset of $G$, then there is some edge $uv$ of $G_0$
such that $(K_u\cup K_v)\setminus S$ is an independent set in $H_{uv}$.
By the properties of $H_{uv}$, this implies that 
$$\Delta_G(S)
\geq \max\Big\{ |K_u\cap S|,|K_v\cap S|\Big\}-1
\geq \Delta-2\left\lceil\sqrt{\Delta}\right\rceil
-c\frac{\log (\Delta')}{\Delta'}n'
\geq \Delta-(c+3)\sqrt{\Delta}\log(\Delta).$$
Instead of $\Delta_G(S)$, one may alternatively consider $\bar{d}_G(S)$.
Nevertheless, in our setting, this only makes sense for cutsets $S$ that are small or minimal.
In fact, if $G$ is a graph of order $n$ and maximum degree $\Delta$,
and $u$ is some vertex in $G$, then $G$ has an independent set $I$ of order at least $\frac{n-\Delta^2-1}{\Delta+1}$ that does not contain any vertex at distance at most two from $u$, 
and $S=N_G(u)\cup I$ is a cutset of $G$ with $\bar{d}_G(S)\to 0$ 
for $n\to \infty$ and fixed $\Delta$.

Our first result improves the trivial bound (\ref{e1}).
All proofs are given in Section \ref{section2}.

\begin{theorem}\label{theorem1}
Let $\Delta$ be an integer at least $3$.
If $G$ is a connected graph of order at least $2\Delta+4$ and maximum degree at most $\Delta$,
then $G$ has a cutset $S$ of order at most $\Delta$ with 
\begin{eqnarray}\label{e2}
\Delta_G(S)&\leq& \Delta-3.
\end{eqnarray}
\end{theorem}
For $\Delta=3$, 
the bound (\ref{e2}) is clearly best possible. 
The results of \cite{lepf} allow to construct
arbitrarily large connected $4$-regular graphs
without independent cutsets,
that is, the bound (\ref{e2}) is also best possible for $\Delta=4$.
See Figure \ref{fig4} for an illustration.

\begin{figure}[H]
\begin{center}
\unitlength 0.6mm 
\linethickness{0.4pt}
\ifx\plotpoint\undefined\newsavebox{\plotpoint}\fi 
\begin{picture}(149,59)(0,0)
\put(5,19){\circle*{2}}
\put(25,19){\circle*{2}}
\put(45,19){\circle*{2}}
\put(65,19){\circle*{2}}
\put(85,19){\circle*{2}}
\put(105,19){\circle*{2}}
\put(125,19){\circle*{2}}
\put(5,39){\circle*{2}}
\put(25,39){\circle*{2}}
\put(45,39){\circle*{2}}
\put(65,39){\circle*{2}}
\put(85,39){\circle*{2}}
\put(105,39){\circle*{2}}
\put(125,39){\circle*{2}}
\put(5,39){\line(0,-1){20}}
\put(25,39){\line(0,-1){20}}
\put(45,39){\line(0,-1){20}}
\put(65,39){\line(0,-1){20}}
\put(85,39){\line(0,-1){20}}
\put(105,39){\line(0,-1){20}}
\put(125,39){\line(0,-1){20}}
\put(5,19){\line(1,0){20}}
\put(25,19){\line(1,0){20}}
\put(45,19){\line(1,0){20}}
\put(65,19){\line(1,0){20}}
\put(85,19){\line(1,0){20}}
\put(105,19){\line(1,0){20}}
\put(25,19){\line(-1,1){20}}
\put(45,19){\line(-1,1){20}}
\put(65,19){\line(-1,1){20}}
\put(85,19){\line(-1,1){20}}
\put(105,19){\line(-1,1){20}}
\put(125,19){\line(-1,1){20}}
\put(5,39){\line(1,0){20}}
\put(25,39){\line(1,0){20}}
\put(45,39){\line(1,0){20}}
\put(65,39){\line(1,0){20}}
\put(85,39){\line(1,0){20}}
\put(105,39){\line(1,0){20}}
\qbezier(5,39)(65,59)(125,39)
\qbezier(5,19)(65,-1)(125,19)
\qbezier(5,19)(149,-27)(125,39)
\end{picture}
\end{center}
\caption{The $4$-regular graph $C_{14}^2$ has no independent cutset.}\label{fig4}
\end{figure}
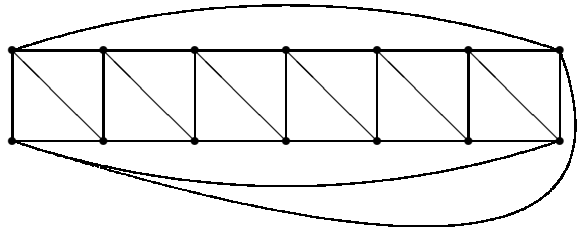
For $\Delta=5$, 
the icosahedron in Figure \ref{fig1} is a $5$-regular graph $G$
in which every cutset $S$ satisfies $\Delta_G(S)\geq 2$.
In fact, note 
that the neighborhood of every vertex of the icosahedron induces a copy of $C_5$,
that every vertex $u$ in every minimal cutset $S$ of the icosahedron 
has neighbors in different components of $G-S$, and 
that the $C_5$ in the neighborhood of $u$ implies that $S$ 
contains at least two neighbors of $u$.

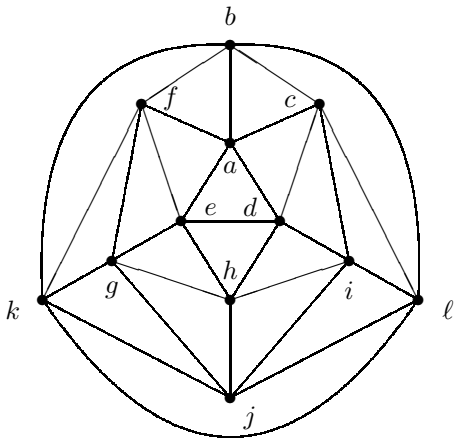
\begin{figure}[H]
\begin{center}
{\small 
\unitlength 1.3mm 
\linethickness{0.4pt}
\ifx\plotpoint\undefined\newsavebox{\plotpoint}\fi 
\begin{picture}(44,43)(0,0)
\put(22,30){\circle*{1}}
\put(22,40){\circle*{1}}
\put(17,22){\circle*{1}}
\put(3,14){\circle*{1}}
\put(41,14){\circle*{1}}
\put(27,22){\circle*{1}}
\put(13,34){\circle*{1}}
\put(31,34){\circle*{1}}
\put(22,40){\line(-3,-2){9}}
\put(13,34){\line(1,-3){4}}
\put(17,22){\line(1,0){10}}
\put(27,22){\line(1,3){4}}
\put(31,34){\line(-3,2){9}}
\put(22,40){\line(0,-1){10}}
\multiput(22,30)(.058064516,.025806452){155}{\line(1,0){.058064516}}
\multiput(27,22)(-.025906736,.041450777){193}{\line(0,1){.041450777}}
\multiput(22,30)(-.025906736,-.041450777){193}{\line(0,-1){.041450777}}
\multiput(22,30)(-.058064516,.025806452){155}{\line(-1,0){.058064516}}
\put(22,14){\circle*{1}}
\multiput(27,22)(-.025906736,-.041450777){193}{\line(0,-1){.041450777}}
\multiput(22,14)(-.025906736,.041450777){193}{\line(0,1){.041450777}}
\put(10,18){\circle*{1}}
\put(34,18){\circle*{1}}
\multiput(13,34)(-.025862069,-.137931034){116}{\line(0,-1){.137931034}}
\multiput(31,34)(.025862069,-.137931034){116}{\line(0,-1){.137931034}}
\put(10,18){\line(3,-1){12}}
\put(34,18){\line(-3,-1){12}}
\multiput(17,22)(-.04516129,-.025806452){155}{\line(-1,0){.04516129}}
\multiput(27,22)(.04516129,-.025806452){155}{\line(1,0){.04516129}}
\put(22,4){\circle*{1}}
\multiput(22,4)(.0259179266,.030237581){463}{\line(0,1){.030237581}}
\multiput(10,18)(.0259179266,-.030237581){463}{\line(0,-1){.030237581}}
\multiput(10,18)(-.04516129,-.025806452){155}{\line(-1,0){.04516129}}
\multiput(34,18)(.04516129,-.025806452){155}{\line(1,0){.04516129}}
\multiput(3,14)(.0492227979,-.0259067358){386}{\line(1,0){.0492227979}}
\multiput(41,14)(-.0492227979,-.0259067358){386}{\line(-1,0){.0492227979}}
\put(13,34){\line(-1,-2){10}}
\put(31,34){\line(1,-2){10}}
\qbezier(22,40)(1.5,41)(3,14)
\qbezier(22,40)(42.5,41)(41,14)
\qbezier(41,14)(22,-14)(3,14)
\put(22,27.5){\makebox(0,0)[cc]{$a$}}
\put(16,34.5){\makebox(0,0)[cc]{$f$}}
\put(28,34.2){\makebox(0,0)[cc]{$c$}}
\put(22,43){\makebox(0,0)[cc]{$b$}}
\put(20,23.2){\makebox(0,0)[cc]{$e$}}
\put(24,23.5){\makebox(0,0)[cc]{$d$}}
\put(10,15){\makebox(0,0)[cc]{$g$}}
\put(22,17){\makebox(0,0)[cc]{$h$}}
\put(34,15){\makebox(0,0)[cc]{$i$}}
\put(24,2){\makebox(0,0)[cc]{$j$}}
\put(0,13){\makebox(0,0)[cc]{$k$}}
\put(44,13){\makebox(0,0)[cc]{$\ell$}}
\put(22,4){\line(0,1){10}}
\end{picture}
}
\end{center}
\caption{Icosahedron}\label{fig1}
\end{figure}
The property of having neighborhoods that induce copies of $C_5$
can be exploited to construct arbitrarily large connected $5$-regular graphs $G$
with no cutset $S$ with $\Delta_G(S)\leq 1$.
Consider, for instance, a cyclic structure based on the pattern shown in Figure \ref{fig2};
no vertex from the middle path can be contained in a cutset $S$ with $\Delta_G(S)\leq 1$,
which implies that no such cutset exists.
Hence, the bound (\ref{e2}) is also best possible for $\Delta=5$.

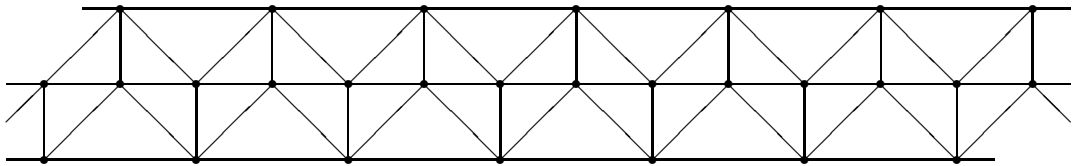
\begin{figure}[H]
\begin{center}
\unitlength 1mm 
\linethickness{0.4pt}
\ifx\plotpoint\undefined\newsavebox{\plotpoint}\fi 
\begin{picture}(145,26)(0,0)
\put(10,5){\circle*{1}}
\put(30,5){\circle*{1}}
\put(50,5){\circle*{1}}
\put(70,5){\circle*{1}}
\put(90,5){\circle*{1}}
\put(110,5){\circle*{1}}
\put(130,5){\circle*{1}}
\put(20,15){\circle*{1}}
\put(40,15){\circle*{1}}
\put(60,15){\circle*{1}}
\put(80,15){\circle*{1}}
\put(100,15){\circle*{1}}
\put(120,15){\circle*{1}}
\put(140,15){\circle*{1}}
\put(10,15){\circle*{1}}
\put(30,15){\circle*{1}}
\put(50,15){\circle*{1}}
\put(70,15){\circle*{1}}
\put(90,15){\circle*{1}}
\put(110,15){\circle*{1}}
\put(130,15){\circle*{1}}
\put(20,25){\circle*{1}}
\put(40,25){\circle*{1}}
\put(60,25){\circle*{1}}
\put(80,25){\circle*{1}}
\put(100,25){\circle*{1}}
\put(120,25){\circle*{1}}
\put(140,25){\circle*{1}}
\put(10,15){\line(1,0){10}}
\put(30,15){\line(1,0){10}}
\put(50,15){\line(1,0){10}}
\put(70,15){\line(1,0){10}}
\put(90,15){\line(1,0){10}}
\put(110,15){\line(1,0){10}}
\put(130,15){\line(1,0){10}}
\put(20,15){\line(1,0){10}}
\put(40,15){\line(1,0){10}}
\put(60,15){\line(1,0){10}}
\put(80,15){\line(1,0){10}}
\put(100,15){\line(1,0){10}}
\put(120,15){\line(1,0){10}}
\put(10,5){\line(1,0){20}}
\put(30,5){\line(1,0){20}}
\put(50,5){\line(1,0){20}}
\put(70,5){\line(1,0){20}}
\put(90,5){\line(1,0){20}}
\put(110,5){\line(1,0){20}}
\put(30,25){\line(1,0){20}}
\put(50,25){\line(1,0){20}}
\put(70,25){\line(1,0){20}}
\put(90,25){\line(1,0){20}}
\put(110,25){\line(1,0){20}}
\put(10,15){\line(1,1){10}}
\put(30,15){\line(1,1){10}}
\put(50,15){\line(1,1){10}}
\put(70,15){\line(1,1){10}}
\put(90,15){\line(1,1){10}}
\put(110,15){\line(1,1){10}}
\put(130,15){\line(1,1){10}}
\put(20,25){\line(0,-1){10}}
\put(40,25){\line(0,-1){10}}
\put(60,25){\line(0,-1){10}}
\put(80,25){\line(0,-1){10}}
\put(100,25){\line(0,-1){10}}
\put(120,25){\line(0,-1){10}}
\put(140,25){\line(0,-1){10}}
\put(10,15){\line(0,-1){10}}
\put(30,15){\line(0,-1){10}}
\put(50,15){\line(0,-1){10}}
\put(70,15){\line(0,-1){10}}
\put(90,15){\line(0,-1){10}}
\put(110,15){\line(0,-1){10}}
\put(130,15){\line(0,-1){10}}
\put(10,5){\line(1,1){10}}
\put(30,5){\line(1,1){10}}
\put(50,5){\line(1,1){10}}
\put(70,5){\line(1,1){10}}
\put(90,5){\line(1,1){10}}
\put(110,5){\line(1,1){10}}
\put(130,5){\line(1,1){10}}
\put(20,15){\line(1,-1){10}}
\put(40,15){\line(1,-1){10}}
\put(60,15){\line(1,-1){10}}
\put(80,15){\line(1,-1){10}}
\put(100,15){\line(1,-1){10}}
\put(120,15){\line(1,-1){10}}
\put(20,25){\line(1,-1){10}}
\put(40,25){\line(1,-1){10}}
\put(60,25){\line(1,-1){10}}
\put(80,25){\line(1,-1){10}}
\put(100,25){\line(1,-1){10}}
\put(120,25){\line(1,-1){10}}
\put(10,15){\line(-1,-1){5}}
\put(30,25){\line(-1,0){15}}
\put(10,15){\line(-1,0){5}}
\put(10,5){\line(-1,0){5}}
\put(130,25){\line(1,0){15}}
\put(140,15){\line(1,0){5}}
\put(140,15){\line(1,-1){5}}
\put(130,5){\line(1,0){5}}
\end{picture}
\end{center}
\caption{Part of a connected $5$-regular $G$ with no cutset $S$ with $\Delta_G(S)\leq 1$.}\label{fig2}
\end{figure}
While the bound (\ref{e2}) is best possible for $\Delta\in \{ 3,4,5\}$,
we can refine it slightly for $5$-regular and $4$-regular graphs.

\begin{theorem}\label{theorem2}
If $G$ is a connected $5$-regular graph of sufficiently large order $n$, 
then there is a cutset $S$ of order at most $5$ with
$\Delta_G(S)\leq 2$
and
$\bar{d}_G(S)<2$.
\end{theorem}
A detailed analysis shows that Theorem \ref{theorem2} holds for $n\geq 14$.

\begin{theorem}\label{theorem3}
If $G$ is a connected $4$-regular graph of sufficiently large order $n$ such that $G[N_G(x)]$ is not isomorphic to $2K_2$ for some vertex $x$ of $G$,
then either $G$ is isomorphic to $C_n^2$
or there is a minimal cutset $S$ of order at most $4$ with $\bar{d}_G(S)<1$.
\end{theorem}
The $4$-regular graphs in which every neighborhood induces a $2K_2$
form a rich class of graphs; the Cartesian product of $K_3$ with itself
and line graphs of cubic triangle-free graphs are examples.

Le, Mosca, and M\"{u}ller \cite{lemomu} conjectured that every
$3$-connected planar graph of maximum degree at most $4$ 
has an independent cutset whose order is bounded by a fixed constant, or it has no independent cutset at all. 
We show two related statements.

\begin{theorem} \label{theorem4}
If $G$ is a $4$-regular graph with connectivity $\kappa\leq 3$, 
then $G$ has an independent cutset of order at most $3$.
\end{theorem}

\begin{theorem} \label{theorem5}
Let $\Delta$ and $r$ be positive integers
such that $c=3+\left\lfloor\frac{2(\Delta-3r+2)}{r(r-1)}\right\rfloor>3$.
If $G$ is a graph of maximum degree $\Delta$
and order more than $\Delta+(c-3)(r-1)+r$ 
that does not contain $K_{r,r}$ as a subgraph,
then $G$ has a cutset $S$ of order at most $\Delta+(c-3)(r-2)$ with 
\begin{eqnarray*}
\Delta_G(S)&\leq& \Delta-c.
\end{eqnarray*}
\end{theorem}
We conclude with a simple consequence of the result of Chen and Yu \cite{chyu}.

\begin{proposition}\label{proposition2}
If $G$ is a connected graph of order $n$, size $m$, and maximum degree $\Delta$
with $m\leq \left(2+\frac{1}{\Delta^2+1}\right)n-4$,
then $G$ has a cutset $S$ with $\Delta_G(S)\leq 1$.
\end{proposition}

\section{Proofs}\label{section2}

As announced we give the proofs of our results.

\begin{proof}[Proof of Theorem~\ref{theorem1}.]
Let $u$ be a vertex of $G$ that is of minimum degree $\delta$.
Clearly, we may assume that $\delta\geq \Delta-1$
because otherwise $S=N_G(u)$ has the desired properties.
Starting with $U_1=\{ u\}$ and $S_1=N_G(u)$, we construct two sequences 
$U_1\subseteq U_2\subseteq U_3\subseteq\ldots\subseteq U_k$
and $S_1,S_2,S_3,\ldots,S_k$ of sets of vertices of $G$ 
such that, for every $i\in [k]$, 
\begin{itemize}
\item $|U_i|=i$, $|S_i|\leq \delta$, 
\item $U_i$ is the vertex set of a component of $G-S_i$, and 
\item every vertex in $S_i$ has a neighbor in $U_i$.
\end{itemize}
Clearly, the sets $S_1$ and $U_1$ have the desired properties.
Let $n_i=|S_i|$ 
and let $m_i$ be the number of edges of $G$ between $S_i$ and $U_i$.
While $\Delta_G(S_i)\geq \Delta-2$, 
we construct $S_{i+1}$ and $U_{i+1}$ in such a way that 
\begin{itemize}
\item either $n_{i+1}=n_i$ and $m_{i+1}\geq m_i+(\Delta-2)$
\item or $n_{i+1}=n_i-1$ and $m_{i+1}\geq m_i+(\Delta-4)$ 
\end{itemize}
as follows: 
\begin{quote}
Let $v\in S_i$ have at least $\Delta-2$ neighbors in $S_i$.
Since $v$ has a neighbor in $U_i$, 
the set $N=N_G(v)\setminus (S_i\cup U_i)$ contains at most one vertex.
Let $S_{i+1}=(S_i\setminus \{ v\})\cup N$ and $U_{i+1}=U_i\cup \{ v\}$.
\end{quote}
It is easy to verify the desired properties for $S_{i+1}$ and $U_{i+1}$.
Once $\Delta_G(S_i)\leq \Delta-3$, the construction terminates and we set $k=i$.
In order to complete the proof, 
we need to argue that the construction actually terminates
and that $R_k=V(G)\setminus (S_k\cup U_k)$ is not empty.

Note that 
\begin{eqnarray*}
m_1-2n_1&=&\delta-2\delta\\
&\geq& -\Delta,\\
m_{k-1}-2n_{k-1}&\leq &\left(\Delta n_{k-1}-2(\Delta-2)\right)-2n_{k-1}\\
&=& (\Delta-2)(n_{k-1}-2)\\
&\leq& (\Delta-2)(\delta-2)\\
&\leq& (\Delta-2)^2,\mbox{ and}\\
m_{i+1}-2n_{i+1}&\geq & (m_i-2n_i)+(\Delta-2)\mbox{ for $i<k$}.
\end{eqnarray*}
Together, these estimates imply
$-\Delta+(\Delta-2)(k-2)\leq (\Delta-2)^2$
and, hence, $k\leq \Delta+1+\frac{2}{\Delta-2}\leq \Delta+3$.
Since 
$|U_k|+|S_k|\leq k+\Delta\leq 2\Delta+3<2\Delta+4$, 
the set $R_k$ is not empty, and 
the set $S_k$ is a cutset of order at most $\Delta$ with $\Delta_G(S_k)\leq \Delta-3$.
\end{proof}

The following property of the icosahedron should be known
but we were unable to find a reference.
Graphs in which all neighborhoods induce isomorphic graphs
were studied by Chan, Fu, and Li \cite{chfuli}.

\begin{proposition}\label{proposition1}
If $G$ is a connected graph such that $G[N_G(u)]$ is isomorphic to $C_5$ 
for every vertex $u$ of $G$, then $G$ is the icosahedron.
\end{proposition}
\begin{proof}
Let $G$ be as in the statement. 

We repeatedly apply the following two observations
for five distinct vertices $u$, $v$, $x$, $y$, and $z$ of $G$:
\begin{itemize}
\item[($O_1$)] If $vxyz$ is a path in the neighborhood of $u$,
then $v$ and $z$ are not adjacent.
\item[($O_2$)] If $xyz$ is a path in the neighborhood of $u$ and $v$ is not adjacent to $x$ or $z$, then $v$ is not adjacent to $u$.
\end{itemize}
Let $a$ be a vertex of $G$.
Let $bcdefb$ be the $C_5$ induced by the neighbors of $a$.
Let $g$ and $h$ be the further two neighbors of $e$ 
such that $afghda$ is the $C_5$ induced by the neighbors of $e$.
By ($O_1$), the vertex $g$ is not adjacent to $b$ and $h$ is not adjacent to $c$.
By ($O_2$), the vertex $g$ is not adjacent to $c$ and $h$ is not adjacent to $b$.
Let $i$ be the further neighbor of $d$
such that $acihea$ is the $C_5$ induced by the neighbors of $d$.
By ($O_1$), the vertex $i$ is not adjacent to $b$ or $g$.
By ($O_2$), the vertex $i$ is not adjacent to $f$.
Let $j$ be the further neighbor of $h$
such that $degjid$ is the $C_5$ induced by the neighbors of $h$.
By ($O_1$), the vertex $j$ is not adjacent to $c$ or $f$.
By ($O_2$), the vertex $j$ is not adjacent to $b$.
Let $k$ be the further neighbor of $g$
such that $efkjhe$ is the $C_5$ induced by the neighbors of $g$.
Since $baegk$ is a path in the neighborhood of $f$, the vertex $k$ is adjacent to $b$.
By ($O_1$), the vertex $k$ is not adjacent to $c$.
By ($O_2$), the vertex $k$ is not adjacent to $i$.
Let $\ell$ be the further neighbor of $i$
such that $cdhj\ell c$ is the $C_5$ induced by the neighbors of $i$.
Since $kghi\ell$ is a path in the neighborhood of $j$, the vertex $\ell$ is adjacent to $k$.
Since $badi\ell$ is a path in the neighborhood of $c$, the vertex $\ell$ is adjacent to $b$.
Since the graph induced by the twelve vertices $a,b,\ldots,\ell$ is $5$-regular,
it equals the icosahedron, which completes the proof.
\end{proof}

\begin{proof}[Proof of Theorem~\ref{theorem2}.]
Let $G$ be as in the statement.
By Proposition \ref{proposition1}, 
there is a vertex $u$ whose neighborhood does not induce a copy of $C_5$.
If $\Delta_G(N_G(u))\leq 2$, then this implies $\bar{d}_G(S)<2$,
and $S=N_G(u)$ has the desired properties.
Hence, we may assume that $\Delta_G(N_G(u))\geq 3$.
Let $U_1\subseteq U_2\subseteq U_3\subseteq\ldots\subseteq U_k$
and $S_1,S_2,S_3,\ldots,S_k$ be exactly as in the proof of Theorem \ref{theorem1},
that is, 
$k\geq 2$,
$5=n_1\geq n_2\geq \ldots \geq n_k$,
$5=m_1<m_2<\ldots<m_k$, 
and $\Delta_G(S_k)\leq 2$.
In view of the desired statement, 
we may assume that $G[S_k]$ is $2$-regular,
that is, the set $S_k$ induces a cycle.
If some vertex $x$ in $S_k$ has no neighbor in $R_k=V(G)\setminus (U_k\cup S_k)$,
then $S_k\setminus \{ x\}$ has the desired properties.
Hence, we may assume that every vertex in $S_k$ has at least one neighbor in $R_k$,
and, hence, at most two neighbors in $U_k$.
Since $m_k>n_k$, some vertex $v$ in $S_k$ has exactly two neighbors in $U_k$,
and, hence, exactly one neighbor $w$ in $V(G)\setminus (U_k\cup S_k)$.
Let $S_{k+1}=(S_k\setminus \{ v\})\cup \{ w\}$
and $U_{k+1}=U_k\cup \{ v\}$.
Using a notation similar to the proof of Theorem \ref{theorem1},
we obtain $n_{k+1}=n_k$ and $m_{k+1}>m_k$.
Possibly proceeding as in the proof of Theorem \ref{theorem1},
we continue the above sequences with
$U_k\subseteq U_{k+1}\subseteq\ldots\subseteq U_{k_2}$
and $S_k,S_{k+1},\ldots,S_{k_2}$ such that 
$5\geq n_k\geq n_{k+1}\geq \ldots \geq n_{k_2}$,
$5<m_k<m_{k+1}<\ldots<m_{k_2}$, and $\Delta_G(S_{k_2})\leq 2$.
Again, we may assume that $G[S_{k_2}]$ is $2$-regular,
and repeat exactly the same arguments as above
to obtain 
$U_{k_2+1}\subseteq \ldots \subseteq U_{k_3}$
and $S_{k_2+1},\ldots,S_{k_3}$.
Since $m_i$ strictly increases but $n_i\leq 5$,
this process can only be repeated a bounded number of times
before it returns a cutset with the desired properties.
This completes the proof.
\end{proof}

\begin{proof}[Proof of Theorem~\ref{theorem3}.]
Let $G$ be as in the statement.
We call a minimal cutset $S$ of $G$ of order at most $4$ with $\bar{d}_G(S)<1$ {\it good},
and we assume that $G$ has no good cutset. 
Clearly, this implies that $G$ is $2$-connected.

We establish a series of claims.

\begin{claim}\label{claim1}
$G$ is $3$-connected.
\end{claim}
\begin{proof}[Proof of Claim \ref{claim1}.]
Suppose, for a contradiction, that $S_1=\{ a_1,b_1\}$ is a cutset of $G$.
Since $S_1$ is not good, the vertices $a_1$ and $b_1$ are adjacent.
Let $C_1$ be the vertex of a component of $G-S_1$ 
such that $a_1$ has only one neighbor $a_2$ in $C_1$.
Note that $b_1$ has at most two neighbors in $C_1$.
Since $G$ is $4$-regular, the set $C_1$ contains more than two vertices.
This implies that $S'=\{ a_2,b_1\}$ is a cutset.
Since $S'$ is not good, the vertex $b_1$ is adjacent to $a_2$.
Since $G$ is $2$-connected, the vertex $b_1$ has a neighbor $b_2$ distinct from $a_2$ in $C_1$.
This implies that $S_2=\{ a_2,b_2\}$ is a cutset.
Since $S_2$ is not good, the vertex $a_2$ is adjacent to $b_2$.
Since $G$ is $2$-connected, 
the set $C_2=C_1\setminus \{ a_2,b_2\}$ is the vertex set 
of a component of $G-S_2$.
Since $G$ is $4$-regular, the set $C_2$ contains more than two vertices.
Since $a_2$ is adjacent to $a_1$, $b_1$, and $b_2$, and since $G$ is $2$-connected,
the vertex $a_2$ has exactly one neighbor in $C_2$.
Altogether, the structure around $a_1$, $b_1$, and $C_1$ reproduces oneself within $a_2$, $b_2$, and $C_2$.
Therefore, iteratively repeating the above arguments for $(S_2,C_2)$ instead of $(S_1,C_1)$
yields an infinite sequence $S_1,S_2,\ldots$ of disjoint cutsets,
contradicting the finiteness of $G$.
This completes the proof of the claim.
\end{proof}

\begin{claim}\label{claim2}
$G$ is $4$-connected.
\end{claim}
\begin{proof}[Proof of Claim \ref{claim2}.]
Suppose, for a contradiction, that $G$ is not $4$-connected.
For a cutset $S$ of order $3$ such that $G-S$ has $k$ components,
let $m_1\leq\ldots\leq m_k$ denote the numbers of edges of $G$
between $S$ and the different components of $G-S$, respectively,
and let $m(S)=m_2+\cdots+m_k$.
Now, let $S$ be a cutset of order $3$ such that $m(S)$ is as large as possible.
Since $S$ is not good, we obtain that $\Delta_G(S)\geq 2$.
Let $u$ in $S$ be adjacent to the remaining two vertices in $S$.
It follows that $k=2$ and that $u$ has exactly one neighbor in each component of $G-S$.
Let $C$ be the vertex set of a component of $G-S$ 
such that there are exactly $m_1$ edges between $S$ and $C$.
Since $G$ is $4$-regular, the set $C$ contains more than one vertex.
If $u'$ is the unique neighbor of $u$ in $C$, 
then $S'=(S\setminus \{ u\})\cup \{ u'\}$ is a cutset.
It follows that $m(S')\geq m(S)+2$,
which contradicts the choice of $S$,
and completes the proof of the claim.
\end{proof}

\begin{claim}\label{claim3}
$G$ contains $P_6^2$ as an induced subgraph.
\end{claim}
\begin{proof}[Proof of Claim \ref{claim3}.]
Let $x$ be a vertex of $G$ such that $G[N_G(x)]$ is not isomorphic to $2K_2$.
Since the cutset $N_G(x)$ is not good, and $G$ is $4$-connected,  
some vertex $y$ in $S$ has exactly two neighbors $b$ and $c$ in $S$ 
and exactly one neighbor $d$ outside of $N_G[x]$.
Let $a$ denote the vertex in $S$ distinct from $y$, $b$, and $c$.
Let $S=(N_G(x)\setminus \{ y\})\cup \{ d\}=\{ a,b,c,d\}$.
Note that there are six edges between the cutset $S$ and the component of $G-S$ that contains $x$.

Suppose, for a contradiction, that $\Delta_G(S)\geq 2$.
In this case, proceeding as in the proof of Theorem \ref{theorem1}
while using that $n$ is sufficiently large and that $G$ is $4$-connected
yields a cutset $S'$ of order $4$ with $\Delta_G(S')\leq 1$ 
such that there are at least eight edges between $S'$ 
and the component of $G-S'$ that contains $x$.
Since $S'$ is not good, the graph $G[S']$ is isomorphic to $2K_2$.
Let $S'=\{ a',b',c',d'\}$ and let $a'b'$ and $c'd'$ be the two edges within $S'$.
Since $G$ is $4$-connected, there are exactly two components in $G-S'$, 
and there are exactly four edges between $S$ and 
the vertex set $C'$ of the component of $G-S'$ that does not contain $x$.
Since $n$ is sufficiently large, we may assume that $C'$ contains at least four vertices.
Since $G$ is $4$-connected, 
a simple application of Hall's Theorem implies that the four edges between $S'$ and $C'$ form a matching,
say $a'a''$, $b'b''$, $c'c''$, and $d'd''$.
Now, the set $\{ a',b'',c',d''\}$ is a good cutset, which is a contradiction.
This contradiction implies that $\Delta_G(S)\leq 1$.
Since $S$ is not good, there are exactly two edges within $S$.
By symmetry, these two edges are either $ad$ and $bc$ or $ab$ and $cd$.

Suppose, for a contradiction, that $ad$ and $bc$ are the two edges within $S$.
Since $G$ is $4$-connected, 
a simple application of Hall's Theorem implies the existence of a matching 
containing four edges between $N_G(x)$ and the component of $G-N_G(x)$ that does not contain $x$,
say $aa^{(3)}$, $bb^{(3)}$, $cc^{(3)}$, and $yd$.
Since $a$ is adjacent to $x$, $a^{(3)}$, and $d$, 
we may assume, by symmetry, that $a$ is not adjacent to $b^{(3)}$.
Now, the set $\{ a,b^{(3)},c,y\}$ is a good cutset, which is a contradiction.
This contradiction implies that $ab$ and $cd$ are the two edges within $S$.
We obtain that $G$ contains the square of the path $abxycd$ as an induced subgraph,
which completes the proof of the claim.
\end{proof}

\begin{claim}\label{claim4}
If $G$ contains $P_{n'}^2$ as an induced subgraph 
for some $n'$ with $6\leq n'<n$, 
then either $G$ contains $P_{n''}^2$ as an induced subgraph 
for some $n''$ with $n'<n''<n$
or $G$ is isomorphic to $C_n^2$.
\end{claim}
\begin{proof}[Proof of Claim \ref{claim4}.]
Let $P:aba_2b_2\ldots cd$ be a path of order $n'$ such that $G$ contains $P^2$ as an induced subgraph.
Clearly, the set $S=\{ a,b,c,d\}$ is a cutset.
Since $G$ is $4$-connected, the graph $G-S$ has exactly two components
and there are exactly six edges between $S$ and 
the vertex set $C$ of the component of $G-S$ that does not intersect $P$.
Since $G$ is $4$-regular, the set $C$ contains more than one vertex.

First, suppose that $C$ contains exactly two vertices $u$ and $v$, 
that is, $n'=n-2$.
Since there are six edges between $S$ and $C$, the vertices $u$ and $v$ are adjacent.
Since $a$ and $d$ both have only two neighbors in $P^2$, they are both adjacent to $u$ and $v$.
By symmetry, we may assume that $b$ is adjacent to $u$ and $c$ is adjacent to $v$.
Now, the graph $G$ is isomorphic to the square of the cycle $aba_2b_2\ldots cdvua$.
Hence, we may assume that $C$ contains more than two vertices.

Next, suppose that $C$ contains exactly three vertices $u$, $v$ and $w$, that is, $n'=n-3$.
Since there are six edges between $S$ and $C$, the vertices $u$, $v$, and $w$ form a triangle.
If $a$ and $d$ have the same two neighbors in $C$, say $u$ and $v$,
then $w$ is adjacent to $b$ and $c$, and $N_G(w)$ is a good cutset,
which is a contradiction.
Hence, by symmetry, we may assume that $a$ is adjacent to $u$ and $v$,
and that $d$ is adjacent to $v$ and $w$.
If $u$ is adjacent to $c$ and $w$ is adjacent to $b$, 
then the set $\{ a_2,b_2,u,d\}$ is a good cutset, which is a contradiction.
It follows that $u$ is adjacent to $b$ and $w$ is adjacent to $c$.
Now, the graph $G$ is isomorphic to the square of the cycle $aba_2b_2\ldots cdwvua$.
Hence, we may assume that $C$ contains more than three vertices.

Since $G$ is $4$-connected, 
Hall's Theorem implies that 
there is a matching of size four between $S$ and $C$,
say $aa'$, $bb'$, $cc'$, and $dd'$.
Since the two cutsets $\{ a,b',c,d\}$ and $\{ a,b,c',d\}$ are not good,
it follows that either $a$ is adjacent to $c'$ and $d$ is adjacent to $b'$
or $a$ is adjacent to $b'$ and $d$ is adjacent to $c'$.
If $a$ is adjacent to $c'$ and $d$ is adjacent to $b'$,
then $\{ a,b',c,d'\}$ is a good cutset, which is a contradiction.
It follows that $a$ is adjacent to $b'$ and $d$ is adjacent to $c'$.
Since the cutset $\{ a',b',c,d\}$ is not good,
the vertex $a'$ is adjacent to $b'$.
Now, the square of the path $P':a'b'aba_2b_2\ldots cd$
is an induced subgraph of $G$ and the order of $P'$ is $n'+2<n$,
which completes the proof of the claim.
\end{proof}
Now, an inductive argument using Claim \ref{claim3} for the base case
and Claim \ref{claim4} for the inductive step implies that $G$ is isomorphic to $C_n^2$,
which completes the proof.
\end{proof}

\begin{proof}[Proof of Theorem~\ref{theorem4}]
	Let $G$ be as in the statement with connectivity $\kappa$. The statement is trivial for $\kappa\leq1$. Now, let $S$ be a cutset of order $\kappa \geq 2$ in $G$ that minimizes the order of a smallest component $C$ of $G-S$. If the set $S$ is independent, then $S$ is a desired cutset. So we may assume that it is not. Since $G$ is $4$-regular and $\kappa \leq 3$, there are at least $2$ vertices in the component $C$. This, together with the choice of $S$, implies that every vertex of $S$ has exactly two neighbors in $C$. Thus, there is exactly one edge $uv$ in $G[S]$, and $G-S$ has exactly two components. Let $C'$ be the other component. By Hall's Theorem, there is a matching $M$ of size $\kappa$ between $S$ and $C'$. Let $uu'$ and $vv'$ be two edges of $M$. If there is a third vertex in $S$, then it cannot be adjacent to both $u'$ and $v'$. Therefore we may assume that $u$ is the only neighbor of $u'$ in $S$. We conclude by stating that $(S \setminus \{u\}) \cup \{u'\}$ is an independent cutset.
\end{proof}
If $G$ is a graph of minimum degree $\delta$, 
maximum degree $\Delta$, and
connectivity $\kappa$ strictly smaller than $\delta$, 
then the same argument shows the existence of a cutset $S$ 
with $\Delta_S(G) \leq \Delta - 3$ 
without a further condition on the order of $G$.

\begin{proof}[Proof of Theorem~\ref{theorem5}]
Let $\Delta$, $r$, $c$, and $G$ be as in the statement,
and call a cutset $S$ as in the statement {\it good}.
We assume that $G$ has no good cutset
and show that $G$ contains $K_{r,r}$ as a subgraph.
Therefore, 
by an inductive argument, 
for $i\in \{ 1,\ldots,r\}$,
we show the existence of a cutset $S_i$ 
such that 
\begin{itemize}
\item $S_i$ has order at most $\Delta+(c-3)(i-1)$,
\item one component of $G-S_i$ with vertex set $C_i$ 
has order exactly $i$,
\item every vertex in $S_i$ has a neighbor in $C_i$,
and 
\item there is a subset $T_i$ of $S_i$ containing at least 
$$\Delta-(c-3)\left(\frac{i^2}{2}-\frac{i}{2}\right)-2(i-1)$$
vertices such that every vertex in $C_i$ 
is adjacent to every vertex in $T_i$.
\end{itemize}
Note that $|S_i|+|C_i|\leq \Delta+(c-3)(i-1)+i\leq 
\Delta+(c-3)(r-1)+r$, which is less than the order of $G$.

For $i=1$, let $S_1=T_1$ be the neighborhood 
of some vertex $u_1$ of maximum degree, which forms $C_1$.
Clearly, $S_1$, $T_1$, and $C_1$ satisfy the desired properties.
Now, 
suppose that $S_{i-1}$, $T_{i-1}$, and $C_{i-1}$ 
with the desired properties
have been constructed for some $i\in \{ 2,\ldots,r\}$.
Since $G$ has no good cutset,
some vertex $u_i$ in $S_{i-1}$ 
has at least $\Delta-c+1$ neighbors in $S_{i-1}$.
Since $u_i$ has a neighbor in $C_{i-1}$,
this implies that the set $N_i$ of neighbors of $u_i$ 
outside of $S_{i-1}\cup C_{i-1}$
has order at most $c-2$.
Let 
$$S_i=(S_{i-1}\setminus \{ u_i\})\cup N_i,\,\,\,\,
C_i=C_{i-1}\cup\{ u_i\},\mbox{ and }\,\,\,\,
T_i=T_{i-1}\cap N_G(u_i).$$
Clearly, $|S_i|\leq |S_{i-1}|+(c-3)\leq \Delta+(c-3)(i-1)$,
$C_i$ is the vertex set of some component of $G-S_i$ 
of order exactly $i$, and 
every vertex in $S_i$ has a neighbor in $C_i$.
Since $u_i$ has at most
$|S_{i-1}|-|T_{i-1}|$ neighbors in $S_{i-1}\setminus T_{i-1}$,
we obtain
\begin{eqnarray*}
|T_i| & \geq & \Big(\Delta-c+1\Big)-\Big(|S_{i-1}|-|T_{i-1}|\Big)\\
&\geq & \Big(\Delta-c+1\Big)-\Big(\Delta+(c-3)(i-2)\Big)+
\left(\Delta-(c-3)\left(\frac{(i-1)^2}{2}-\frac{(i-1)}{2}\right)-2(i-2)\right)\\
&=& \Delta-(c-3)\left(\frac{i^2}{2}-\frac{i}{2}\right)-2(i-1)
\end{eqnarray*}
Altogether, we obtain $S_i$, $T_i$, and $C_i$ 
with the desired properties.
The definition of $c$ implies that $|T_r|\geq r$.
Hence, $G[C_r\cup T_r]$ contains $K_{r,r}$ as a subgraph,
which completes the proof.
\end{proof}

\begin{proof}[Proof of Proposition~\ref{proposition2}.]
Let $G$ be as in the statement.
The square of $G$ has an independent set $\{ u_1,\ldots,u_{\alpha}\}$ 
of order $\alpha\geq \frac{n}{\Delta^2+1}$.
Clearly, we may assume that $\Delta_G(N_G(u_i))\geq 2$ for every $i\in [\alpha]$.
Hence, for every $i\in [\alpha]$, there is a neighbor $v_i$ of $u_i$
such that $u_i$ and $v_i$ have at least two common neighbors.
Let the graph $G'$ with $n'$ vertices and $m'$ edges 
arise from $G$ by contracting the edges $u_1v_1,\ldots,u_{\alpha}v_{\alpha}$.
Note that 
\begin{eqnarray*}
m' 
& \leq & m-3\alpha
\leq \left(2+\frac{1}{\Delta^2+1}\right)n-4-3\alpha
\leq 2n-2\alpha-4
= 2n'-4.
\end{eqnarray*}
By the result of Chen and Yu \cite{chyu},
the graph $G'$ has an independent cutset $S'$.
Uncontracting the edges $u_1v_1,\ldots,u_{\alpha}v_{\alpha}$ yields 
a cutset $S$ with the desired properties.
\end{proof}

\end{document}